\documentclass[10pt]{amsart}
\usepackage{amsmath, amscd, amssymb}
\usepackage{graphpap, color}
\usepackage[mathscr]{eucal}
\usepackage{mathrsfs}
\usepackage{pstricks}
\usepackage{color}
\usepackage{cancel}
\usepackage[mathscr]{eucal}
\usepackage{stmaryrd}

\usepackage[pagebackref]{hyperref}

\usepackage[T1]{fontenc}
\usepackage{textcomp}

\usepackage{cite}
\usepackage[all,cmtip]{xy}

\numberwithin{equation}{section}

\newcommand{\CC}{\mathbb{C}}

\newcommand{\RR}{\mathbb{R}}
\newcommand{\ZZ}{\mathbb{Z}}

\def\cln{\colon}

\newcommand{\cal}{\mathcal}

\def\cE{{\cal E}}

\def\cO{{\cal O}}
\def\cP{{\cal P}}

\def\cX{{\cal X}}

\def\begeq{\begin{equation}}
\def\endeq{\end{equation}}
\def\and{\quad{\rm and}\quad}

\def\and{\quad\text{and}\quad}

\newtheorem{prop}{Proposition}[section]
\newtheorem{theo}[prop]{Theorem}
\newtheorem{lemm}[prop]{Lemma}
\newtheorem{coro}[prop]{Corollary}

\newtheorem{defi}[prop]{Definition}

\newtheorem{defi-prop}[prop]{Definition-Proposition}
\newtheorem{quest}[prop]{Question}

\def\dbar{\overline{\partial}}

\def\beq{\begin{equation}}
\def\eeq{\end{equation}}

\def\bee{\begin{equation}}
\def\eeq{\end{equation}}

\begin{document}
\title{Locally rigid implies globally rigid in K\"ahler geometry}

\thanks{The author is supported by NSFC (No. 12271073 and 12271412).}
\thanks{Keywords: deformation rigidity, isotrivial,  smooth family, K\"ahler morphism}
\thanks{MSC(2010): 14D15, 14E30, 14J10}

\author{Mu-lin Li}
\address{School of Mathematics, Hunan University, China}
\email{mulin@hnu.edu.cn}

\date{\today}

\begin{abstract}
In this paper, we study the rigidity properties of compact K\"ahler manifolds. Given a smooth family of compact K\"ahler manifolds $\cX$ over the unit disk $\Delta$, we show that all the fibers are mutually isomorphic if the family is locally trivial at a point $t_1\in\Delta$ and the fiber $\cX_{t_1}$ is non-uniruled.  This proves that the locally rigid implies the global rigid in K\"ahler. It can also be used to prove the so called global non-deformability for non-uniruled K\"ahler manifolds under K\"ahler morphisms.
\end{abstract}
\maketitle

\section{introduction}
Two compact complex manifolds $X$ and $X'$ are direct deformations of each other if and only if there is a proper smooth holomorphic map
\beq\nonumber
\pi:\cX\to\Delta,
\eeq
where $\Delta\subset \CC$ is the unit disk, and where $X$, respectively $X'$ are isomorphic to fibers of $\pi$.  Two compact complex manifolds $X$ and $X'$ are deformation equivalent if and only if there is a sequence of compact complex manifolds $\{X_i\}_{i=0}^k$ such that $X_0=X$, $X_k=X'$ and each $X_i$ is a direct deformation of $X_{i-1}$. Equivalently, $X$ and $X'$ are deformation equivalent if there exists a proper smooth morphism $f:\cX\to B$ between complex manifolds such that there exist two points $b, b'\in B$ with $\cX_b\cong X$ and $\cX_{b'}\cong X'$.

For a compact complex manifold we have the following three types of rigid
\begin{itemize}
\item A compact complex manifold $X$ is said to be globally rigid if for every compact complex manifold $X'$ that is deformation equivalent to $X$, we have an isomorphism $X\cong X'$.
\item A compact complex manifold $X$ is said to be locally rigid if for each deformation of $X$,
\beq\nonumber
\pi:\cX\to B
\eeq
there is an open neighbourhood  $U\subset B$ of $b$ such that $\cX_t=\pi^{-1}(t)\cong X$ for all $t\in U$.
\item A compact complex manifold $X$ is said to be infinitesimally rigid if
\beq\nonumber
H^1(X,T_X)=0,
\eeq
where $T_X$ is the tangent bundle of $X$.
\end{itemize}
From the Kodaira-Spencer-Kuranishi theory we know that if $X$ is infinitesimally rigid, then $X$ is also locally rigid. Because the existence of a moduli space for minimal surfaces of general type \cite{Gi77}, the authors of \cite[Theorem 2.7]{BC} proved that locally rigidity implies globally rigidity when $X$ is a minimal surface of general type. It is natural to ask the following question.
\begin{quest}
Suppose that $X$ is a compact non-uniruled K\"ahler manifold. Is it true that
$X$ is globally rigid whenever it is locally rigid?
\end{quest}

In this paper we consider the above question in K\"ahler geometry. First we introduce the following definitions

Two compact complex manifolds $X$ and $X'$ are direct deformations of each other in K\"ahler if and only if there is a proper smooth holomorphic map
\beq\nonumber
\pi:\cX\to\Delta,
\eeq
where $\Delta\subset \CC$ is the unit disk, and where $X$, respectively $X'$, are isomorphic to fibers of $\pi$ with  all the fibers $\pi^{-1}(t)$ being K\"ahler manifolds.  Two compact complex manifolds $X$ and $X'$ deformation equivalent  in K\"ahler  if and only if there is a sequence of compact complex manifolds $\{X_i\}_{i=0}^k$ such that $X_0=X$, $X_k=X'$ and each $X_i$ is a direct deformation of $X_{i-1}$ in K\"ahler.
\begin{defi}A compact complex K\"ahler manifold $X$ is said to be globally rigid in K\"ahler if for every compact complex K\"ahler manifold $X'$ that is deformation equivalent to $X$ in K\"ahler, we have an isomorphism $X\cong X'$.
\end{defi}
With the help of variations of Hodge structures, we prove the following theorem
\begin{theo}\label{res1} Let $\cX$ be a complex manifold, and let $\pi\cln \cX\to \Delta=\{t\in\CC||t|<1\}$ be a proper smooth holomorphic map such that all the fibers are K\"ahler manifolds. Suppose that the family is locally trivial at a  point $t_1\in\Delta$(or there exists a small neighborhood $V_{t_1}$ of $t_1\in\Delta$ such that $\pi^{-1}(V_{t_1})\cong \cX_{t_1}\times V_{t_1}$). Then  $\cX_t\cong \cX_{t_1}$ for all $t\in\Delta$, provided that  $\cX_{t_1}$ is a non-uniruled K\"ahler manifold.
\end{theo}

\begin{coro}
Suppose that $X$ is a compact non-uniruled K\"ahler manifold. Then $X$ is globally rigid in K\"ahler if it is locally rigid.
\end{coro}

Since deformations of K\"ahler surfaces remain K\"ahler, so the deformation equivalence in K\"ahler coincides with  the ordinary deformation equivalence.  Thus we have
\begin{coro}
Suppose that $X$ is a compact non-uniruled K\"ahler surface. Then $X$ is globally rigid if it is locally rigid.
\end{coro}



Theorem \ref{res1} is also related to the following so called global non-deformability problem

\begin{quest}\label{quest}
Let $\pi\cln \cX\to \Delta=\{t\in\CC||t|<1\}$ be a smooth family of compact complex manifolds. Suppose that all the fibers $\cX_t$ with $t\in\Delta^*=\Delta\setminus  \{0\}$  are biholomorphic to a given compact complex manifold $S$. When  is the central fiber $\cX_0$  also biholomorphic to $S$, that is, $\cX_0\cong S$?
\end{quest}

 For projective manifolds with negative Kodaira dimension, there are many results on  global non-deformability. First of all, it is well-known that Hirzebruch surfaces are not global non-deformability.
  For the affirmative results, Siu \cite{Siu0}, Hwang \cite{Hw}, Hwang and  Mok \cite{HN1,HN2,HN3} made important contributions. For example, Siu \cite{Siu0} proved the global non-deformability for complex projective space, Hwang \cite{Hw} proved the global non-deformability for the complex hyperquadric,  Hwang and  Mok \cite{HN1,HN2,HN3} obtained the global non-deformability  for rational homogeneous spaces $S$  with second Betti number $b_2(S)=1$, under the assumption that the morphism $\pi$ is a K\"ahler morphism. The author \cite{L1},  The author-Liu \cite[Theorem 1.2]{LL} and  The author-Rao-Wang \cite{LRW} provided some results on global non-deformability under K\"ahler morphism.


\begin{coro}\label{res2} Let $\cX$ be a complex manifold, and let $\pi\cln \cX\to \Delta=\{t\in\CC||t|<1\}$ be a proper smooth holomorphic map such that all the fibers are K\"ahler manifolds. Suppose that all the fibers $\cX_t$ are biholomorphic to a given compact non-uniruled K\"ahler manifold $S$ for $t\in\Delta^*=\Delta\setminus  \{0\}$. Then the central fiber $\cX_0$ is also biholomorphic to $S$.
\end{coro}


The main idea of the proof of our main result, Theorem \ref{res1},  is to use the fact that the coase moduli space $M$ of polarized non-uniruled K\"ahler manifolds is Hausdorff. Our major task is  the following two parts: one is to construct a class $\alpha\in H^2(\cX,\RR)$ such that its restriction satisfies $\alpha_t\in H^{1,1}(\cX_t,\CC)$ for all $t\in\Delta$ and $\alpha_{t_1}$ is a K\"ahler class. The other is to use Demailly-P$\breve{\textrm{a}}$un \cite{DP04} to show that  there exists a connected open dense subset $W\subset \Delta$ contains  $U_{t_1}$ such that the restriction $\alpha_t\in H^{1,1}(\cX_t,\CC)$ are K\"ahler classes for for all $t\in W$. Thus the pair $(\pi_{W}:\cX_W=\pi^{-1}(W)\to W ,\alpha)$ forms a family of polarized K\"ahler manifolds.

{\bf Acknowledgement:}  The authors would like to thank Xiao-Lei Liu and Sheng Rao  for their interest and useful discussions.

\section{Preliminaries:Moduli of polarized K\"ahler manifolds, Barlet Spaces}

\subsection{Moduli of polarized K\"ahler manifolds}
\begin{defi}A polarized compact K\"ahler manifold is a pair $(X, \omega_{X})$ consisting of a connected compact complex manifold $X$ and a K\"ahler class $\omega_{X}\in H^2(X,\RR)$(i.e. a class represented by a K\"ahler form).
\end{defi}

An isomorphism of two polarized compact K\"ahler manifolds $(X, \omega_{X})$ and $(\widetilde{X},\omega_{\widetilde{X}})$ is a holomorphic isomorphism $\phi:X\to \widetilde{X}$ with $\phi^*(\omega_{\widetilde{X}})=\omega_{X}$.

\begin{defi}A polarized family of compact K\"ahler manifolds (parametrized by a complex space $B$) is a pair $(f,\omega_{\cX/B})$ consisting of a proper smooth holomorphic map $f:\cX\to B$ of complex spaces with connected fibers and an element
 $\omega_{\cX/B}\in \Gamma(B,R^2f_*\RR)$ such that
 \begin{enumerate}
 \item The value of the section $\omega_{\cX_b}=(\omega_{\cX/B})(b)$ on $b$ is a K\"ahler class on each fiber $\cX_b =f^{-1}(b)$.
  \item $\eta(\omega_{\cX/B})=0$ where $\eta:\Gamma(B,R^2f_*\RR)\to\Gamma(B,R^2f_*\mathcal{O}_{\cX})$ is the homomorphism induced by the natural inclusion $\RR\to \mathcal{O}_{\cX}$ of sheaves on $\mathcal{X}$.
 \end{enumerate}
\end{defi}
When $B$ is a smooth complex manifold, the above definition is related to the following definition of K\"ahler morphism.

\begin{defi}\label{defi-1} A smooth proper holomorphic map $f:\cX\to B$  between complex manifolds is called a  {\emph {K\"ahler morphism}}, if there exists an open covering $\{U_{i}\}$ of $\cX$ and smooth functions $\varphi_i$ on $U_i$ such that $\varphi_i-\varphi_j$ is the real part of some holomorphic function on $U_i\cap U_j$, and the forms $\sqrt{-1}\partial\dbar\varphi_i$ are positive definite on $T_{\cX/B}|_{U_i}$.


We say that the morphism $f$ is {\emph { K\"ahler at a point}} $s\in B$, if $f$ becomes a  K\"ahler morphism after shrinking $B$ to some neighborhood of $s$.
\end{defi}

For the K\"ahler morphism $f$ in the above definiton, it is easy to see that the differential form $\omega=\{\omega_i\}$ where $\omega_i=\omega|_{U_i}=\sqrt{-1}\partial\dbar\varphi_i$  is a relative K\"ahler metric, i.e., $\omega|_{\cX_t}$ is a K\"ahler metric on $\cX_t$ for each $t\in B$. In particular,   the form $\omega$ is a K\"ahler metric on $\cX$ if $B$ is a point.

Let $\mathcal{M}$ be the groupoid fibered category over the category of reduced complex spaces, such that for a reduced complex space $B$ the groupoid $\mathcal{M}(B)$ is as follows
 \begin{align*}
 \mathcal{M}(B)=&\bigg\{(f: \cX\to B,\omega_{\cX/B})\, \mbox{is a family of polarized K\"ahler manifolds };\\
&(\cX_b,\omega_{\cX_b})\, \mbox{is a polarized non-uniruled K\"ahler manifold,\,  $\forall b\in B$}\bigg\}.
\end{align*}

 The arrows of $\mathcal{M}(B)$ from $(f: \cX\to B,\omega_{\cX/B})$ to $(f': \cX'\to B,\omega_{\cX'/B})$ are isomorphisms $\Phi:  \cX\to \cX'$ such that $f=f'\circ\Phi $ and $\Phi^*(\omega_{\cX'/B})=\omega_{\cX/B}$. By \cite[Section 1, Theorem]{Fu84} and \cite[(1.3)Theorem]{Sch84}, there exists a coarse separated moduli complex space $M$ for $\mathcal{M}$.

\subsection{Relative cycle spaces}

To construct K\"ahler classes on the  fibers, we need to introduce some notation of the Barlet cycle spaces of compact complex manifolds.

 Let $\pi:\cX\to \Delta$ be a smooth family of $n$-dimensional K\"ahler manifolds.  For $0\le k\le n$, let $C^k(\cX/\Delta)$ be the relative Barlet space, which parametrizes the $k$-dimensional cycles in $\cX$ over $\Delta$. Then there exists a canonical morphism $$\mu:C(\cX/\Delta):=\bigcup_{0\le k\le n}C^k(\cX/\Delta)\to \Delta,$$
 and we have the following properness  property of the morphism $\mu$.

\begin{prop}\cite[Proposition 3.9]{LL24}\label{prop-1}
Let $\pi:\cX\to \Delta$ be a smooth proper morphism whose fibers $\cX_t$ are K\"ahler for all $t\in \Delta$, and $A$ be a connected component of $C(\cX/\Delta)$.  Then the restricted morphism $\mu|_A:A\to \Delta$ is proper.
\end{prop}

Using this properness of connected components of the Barlet cycle space, we have the following Lemma which will be used in the proof of the main theorem.
\begin{lemm}\label{Kahler}
Let $\pi:\cX\to \Delta$ be a smooth proper morphism whose fibers $\cX_t$ are K\"ahler for all $t\in \Delta$. Suppose that there exists a small neighborhood $U_{t_1}$ of $t_1\in\Delta$ such that $\cX_t\cong \cX_{t_1}$ for $t\in U_{t_1}$. Let $A$ be a connected component of $C(\cX/\Delta)$ satisfying $A\times_{\Delta} U_{t_1}\neq \emptyset$. Then the image $\mu(A)=\Delta$.
\end{lemm}
\begin{proof}
  Let $\pi_{U_{t_1}}: \cX_{U_{t_1}}:=\pi^{-1}(U_{t_1})\to U_{t_1}$ be the restriction of $\pi$ on $\cX_{U_{t_1}}$. Since the fibers $\cX_t\cong \cX_{t_1}$ for all $t\in U_{t_1}$, the smooth family $\pi_{U_{t_1}}$ is locally trivial by \cite{FG}. Thus, for each $t\in U_{t_1}$, there exists a small neighborhood $V_t\subset U_{t_1}$ of $t$ such that
$$\pi^{-1}(V_t)\cong V_t\times \cX_{t_1}.$$
  By the functorial property of the relative Barlet space, we have
$$
 C(\pi^{-1}(V_t)/V_t)\cong C(\cX_{t_1})\times V_t.
$$
 Because $A\times_{\Delta} U_{t_1}\neq \emptyset$, there exists a point $\widetilde{t}\in U_{t_1}$ such that
 \beq\nonumber
 V_{\widetilde{t}} \times_{\Delta}A \subset C(\pi^{-1} (V_{\widetilde{t}})/V_{\widetilde{t}})\cong C(\cX_{\tilde{t}})\times V_{\tilde{t}}
 \eeq
 is a connected component. Thus the projection $V_{\widetilde{t}} \times_{\Delta}A\to V_{\tilde{t}}$ is surjective. Note that $\mu|_{A}$ is a proper morphism, so $\mu(A)\subseteq \Delta$ is an analytic subset containing the open subset $V_{\tilde{t}}$. Hence $\mu(A)$ must be $\Delta$.
\end{proof}

We also need the following celebrated theorem proved by Demailly-P\u{a}un \cite{DP04} and Collins-Tosatti \cite{CT15}.
\begin{theo}[{\cite[Theorem 0.9]{DP04},\cite[Corollary 1.3]{CT15}}]\label{Kahler-1}
If $X$ is a compact K\"ahler manifold, then the K\"ahler cone $\mathcal{K}(X)$ of $X$ is one of the connected components of the set $\mathcal{P}$ of real $(1, 1)$-cohomology classes $\alpha=\{\omega\}$ which are numerically positive on analytic cycles, i.e. such that $\int_Z\omega^p>0$ for every irreducible analytic set $Z$ in $X$ with $p=\dim Z$.
\end{theo}

\section{Rigidity}
First we recall several basic constructions for variations of Hodge structures from \cite[Chapter 10]{Vo}, \cite[Section II.1]{G68} and \cite[Chapter 4]{CMP}.
 When $\pi\cln \cX\to \Delta$ is a smooth family of compact complex manifolds, by \cite[Theorem 9.3]{Vo} $\pi:\cX\to\Delta$ forms a smooth fiber bundle. There exists a diffeomorphism
\beq
T:\cX\cong \cX_{t_1}\times \Delta
\eeq
such that $pr_2\circ T=\pi$, where $pr_2:\cX_{t_1}\times \Delta\to \Delta$ is the projection, and $T|_{\pi^{-1}(t_1)}=Id:\cX_{t_1}\to \cX_{t_1}$.

 The explanation of the existence of $T$ is from \cite{Math-1}: For the constant map $F_0:\Delta\to \Delta$ which maps the disk to $t_1$, Let $F:\Delta\times [0,1]\to \Delta$ be a smooth homotopy from  $F_0:\Delta\to \Delta$ to the identity map $F_1:\Delta\to \Delta$. The pullback $F^*\cX$ is a smooth fiber bundle over $\Delta\times [0,1]$. By \cite[Theorem 5]{del}, this fiber bundle admits a complete Ehresmann connection.  Let  $v$ be the trivial vector field of $[0,1]$, it induces a vector field $\widetilde{v}$ on $\Delta\times [0,1]$. Denote by $\mathcal{V}$ the horizontal vector field over $F^*\cX$ induced by the pullback of $\widetilde{v}$. Then the flow induced by $\mathcal{V}$ yields an isomorphism between $F_0^*\cX$ and $F^*_1\cX$, which gives the isomorphism $T^{-1}: \cX_{t_1}\times \Delta \to \cX$.

   Let  $T_{t_1}:=pr_1\circ T: \cX\to \cX_{t_1}$ and $\kappa_t: \cX_t\to \cX_{t_1}$ be the composition of $\imath_t: \cX_t\to \cX$ with $T_{t_1}$. The diffeomorphisms $\{\kappa_t\}$
 depend smoothly on $t$. Let $\psi_t:\cX_{t_1}\to \cX_t$ be the inverse diffeomorphism of $\kappa_t$. See the following commutative diagram.
\begin{eqnarray}\label{diagram-1}
\xymatrix{
\cX_t \ar[r]^{\imath_t} \ar@<.5ex>[rd]^{\kappa_t}  &  \cX\ar[d]^{T_{t_1}}\ar[r]^(0.35){T}\ar[rd]^(0.35){\pi}|\hole  & \cX_{t_1}\times\Delta \ar[dl]^(0.6){p_1}\ar[d]^{p_2}\\
& \cX_{t_1}\ar@<.5ex>[lu]^{\psi_t}  & \Delta
}
\end{eqnarray}

Because $\pi:\mathcal{X}\to \Delta$ is a smooth family of compact complex manifolds, we have that
\beq\label{iden-1}
\psi_t^*H^{k}(\cX_t,\ZZ)=H^{k}(\cX_{t_1},\ZZ),\, \psi_t^*H^{k}(\cX_t,\CC)=H^{k}(\cX_{t_1},\CC).
\eeq
Thus
\beq\label{iden-2}E^k:=\bigcup_{t\in\Delta}\psi_t^*H^{k}(\cX_t,\CC)
\eeq
forms a constant sheaf on $\Delta$, and $E^k=\Delta\times H^{k}(\cX_{t_1},\CC)\cong R^k\pi_*\CC$ over $\Delta$. Let $\Omega_{\Delta}$ be the cotangent sheaf of $\Delta$. Following \cite[Definition 9.13]{Vo}, there exists a flat connection
\beq
\nabla^{k,GM}: R^k\pi_*\CC\otimes \cO_{\Delta}\to R^k\pi_*\CC\otimes \Omega_{\Delta}
\eeq
such that the sections $\cup_{t\in\Delta}\kappa_t^*\alpha_{t_1}$ induced by the classes $\alpha_{t_1}\in H^{k}(\cX_{t_1},\CC)$ are flat sections of $\nabla^{k,GM}$.

Because
\beq\nonumber
\psi_t^*H^{2}(\cX_t,\CC)=H^{2}(\cX_{t_1},\CC),
\eeq
the filtration
\beq\nonumber
F^p_t:=\bigoplus_{r+s=2,r\ge p}\psi_t^*H^{r,s}(\cX_t)
\eeq
defines a Hodge filtration on the space $H^{2}(\cX_{t_1},\CC)$ for $t\in\Delta$, where $H^2(\cX_t,\CC)=\sum_{r+s=2} H^{r,s}(\cX_t)$. Since $\Delta$ is contractible, this induces  the following holomorphic morphism as in \cite[(1.1) Theorem]{G68} or \cite[Theorem 10.9]{Vo}
$$
\cP^{p,2}\cln  \Delta\to \mbox{Grass}(b^{p,2},H^2(\cX_{t_1},\CC)),
$$
where $\mbox{Grass}(b^{p,2},H^2(\cX_{t_1},\CC))$ is the  Grassmannian of subspaces in $H^2(\cX_{t_1},\CC)$ with dimension $b^{p,2}$.

\begin{lemm}\label{lemm4}Let $\pi\cln \cX\to \Delta$ be a proper smooth morphism from a complex manifold $\cX$ to the unit disc $\Delta$ such that all the fibers $\cX_t$ are K\"ahler for $t\in \Delta$. Suppose that there exists a small neighborhood $U_{t_1}$ of $t_1\in\Delta$ such that $\cX_t\cong \cX_{t_1}$ for $t\in U_{t_1}$.  Let $A\in H^{1,1}(\cX_{t_0},\CC)$ be a class on $\cX_{t_0}$ for some $t_0\in\Delta$. Then there exists a class $\alpha\in H^2(\cX,\RR)$ such that its restriction satisfies $\alpha_t\in H^{1,1}(\cX_t,\CC)$ for all $t\in\Delta$ and $\alpha_{t_0}=A$.
\end{lemm}
\begin{proof}

Let $\tilde{\pi}=\pi|_{{\pi}^{-1}(U_{t_1})}\cln \pi^{-1}(U_{t_1})\to U_{t_1}$ be the restriction of $\pi$ on $ \pi^{-1}(U_{t_1})$. Since $X_t\cong \cX_{t_1}$ for all $t\in U_{t_1}$, the smooth family $\tilde{\pi}$ is locally trivial by \cite{FG}. Thus, for each $t\in U_{t_1}$, there exists a small neighborhood $V_t\subset U_{t_1}$ of $t$ such that
$$\tilde{\pi}^{-1}(V_t)\cong V_t\times \cX_{t_1}.$$
The Kodaira-Spencer map for $\pi|_{\tilde{\pi}^{-1}(V_t)}$
$$\rho_t\cln T_{V_t,t}\to H^1(\cX_t,T_{\cX_t})$$
is zero, and the differential $d\cP^{p,2}$ is zero on $V_{t}$ by \cite[(1.23) Theorem]{G68} or \cite[Theorem 10.4]{Vo}. Furthermore, we know that $ \cP^{p,2}$ is constant on $V_{t}$. Because $\cP^{p,2}$ is a holomorphic map, the preimage $(\cP^{p,2})^{-1}(\xi)$ is an analytic subset of $\Delta$ for any $\xi\in \mbox{Grass}(b^{p,2},H^2(\cX_{t_1},\CC))$. Since $V_{t}$ is contained in one preimage $(\cP^{p,2})^{-1}(\xi_0)$ by the assumption, the preimage  $ (\cP^{p,2})^{-1}(\xi_0)$ is the whole $\Delta$. So $\cP^{p,2}$ is a constant map.
Therefore
\beq\label{iden-4}
\psi_t^*H^{r,s}(\cX_t)=H^{r,s}(\cX_{t_1})
\eeq
for $t\in \Delta$ and $0\le r,s\le 2,\, r+s=2$. Thus
\beq\nonumber
F^p_t:=\bigoplus_{r+s=2,r\ge p}\psi_t^*H^{r,s}(\cX_t)=F^p_0
\eeq
defines a constant Hodge filtration on the space $H^{2}(\cX_{t_1},\CC)$ for $t\in\Delta$, Denote it by $F^p$.

For a real element $\alpha\in H^{2}(\cX,\CC)$, it induces a global holomorphic section $s_{\alpha}(t):=\overline{\alpha_t}=\overline{\imath_t^*\alpha}$ through the canonical quotient map $E^2\to E^2/F^1$.  The zero locus $Z(s_{\alpha})$ is defined as the set of points $t\in\Delta$ satisfying $\alpha_t\in H^{1.1}(\cX_t,\CC)$. It is an analytic subset of $\Delta$.

Let $\alpha_{t_0}=A\in H^{1,1}(\cX_{t_0},\CC)$ and $\alpha_{t_1}=\psi_{t_0}^*\alpha_{t_0}$. Define $\alpha:=T^*p_1^*\psi_{t_0}^*\alpha_{t_0}=T_{t_1}^*(\psi_{t_0}^*\alpha_{t_0})$. Then
\begin{eqnarray}\label{iden-3}
\psi_t^*\alpha_t&=&\psi_t^*\imath_t^*\left(T_{t_1}^*(\psi_{t_0}^*\alpha_{t_0})\right)\\
&=&\psi_t^*\kappa^*_t\left(\psi_{t_0}^*\alpha_{t_0}\right)\nonumber\\
&=&\psi_{t_0}^*\alpha_{t_0}\nonumber\\
&=&\alpha_{t_1}\nonumber
\end{eqnarray}
 for $t\in \Delta$. Hence the zero locus $Z(s_{\alpha})=\Delta$ by (\ref{iden-4}).
\end{proof}
\subsection{Construction of K\"ahler morphism}we denote by $\cE_{\cX}$ the structure sheaf of complex valued $C^{\infty}$ functions and by $\cE^{p,q}_{\cX}$ the sheaf of smooth $(p,q)$-forms on a manifold $\cX$. $\cE_{\cX,\RR}$ is defined as the subsheaf in $\cE_{\cX}$ of real $\CC$-valued smooth functions on $\cX$.
 Let $PH_{\cX,\RR}$ be the sheaf of pluriharmonic $\RR$-valued functions on $\cX$, and let $\mathcal{K}_{\cX,\RR}:=\cE_{\cX,\RR}/PH_{\cX,\RR}$ be the quotient sheaves. Thus a section $\Theta\in \Gamma(\cX,\mathcal{K}_{\cX,\RR})$ corresponds to an open covering $\{U_i\}$ of $\cX$ with $\varphi_i\in\cE_{X,\RR}(U_i)$ such that $\varphi_i-\varphi_j\in PH_{\cX,\RR}(U_i\cap U_j,)$, and we denote the section $\Theta$ by $\{(U_i,\varphi_i)\}$. Now we set
$$\omega:=\sqrt{-1}\partial\dbar\Theta\in H^2(\cX,\RR)$$ where
\beq\nonumber
\omega|_{U_i}=\sqrt{-1}\partial\dbar\Theta.
\eeq

\begin{lemm}\label{res1-2}
Let $\pi\cln \cX\to \Delta$ be a proper smooth morphism from a complex manifold $\cX$ to the unit disk $\Delta$ such that all the fibers $\cX_t$ are K\"ahler for $t\in \Delta$. Suppose that  $\alpha\in H^2(\cX,\RR)$ satisfies $\alpha_t\in H^{1,1}(\cX_t,\CC)$ for all $t\in\Delta$. Then there exists a section $\Theta\in \Gamma(\cX,\mathcal{K}_{\cX,\RR})$ such that
\begin{eqnarray*}
\{\sqrt{-1}\partial\dbar\Theta\}=\alpha\in H^2(\cX,\RR).
\end{eqnarray*}
\end{lemm}
\begin{proof}

By the definition of $PH_{\cX,\RR}$, we have
\beq\label{real-short}
\begin{CD}
0@>>>  \RR @>>> \cO_\cX@>\text{2Im}>>PH_{\cX,\RR}@>>>0.
\end{CD}
\eeq
The Leray spectral sequence yields the following commutative diagram
\beq\nonumber
\begin{CD}
 @>>>H^1(\cX,PH_{\cX,\RR})@>>>H^2(\cX,\RR) @>^{\varphi}>> H^2(\cX,\cO_\cX)\\
 @.@.@VVV@V^{\cong}VV\\
 @.@.H^0(\Delta,R^2\pi_*\RR)@>^{\widetilde{\varphi}}>>H^0(\Delta,R^2\pi_*\cO_\cX).
\end{CD}
\eeq
The sheaves $R^2\pi_*\RR$ and $R^2\pi_*\cO_\cX$ are locally free, since $\cX_t$ are K\"ahler manifolds for all $t\in\Delta$. Then by base change theorem, we have
 \beq\nonumber
\begin{CD}
\iota_t^* R^2\pi_*\RR @>^{\cong}>>  H^2(\cX_t,\RR) \\
 @V^{\varphi_t}VV@V^{\rho_t}VV\\
 \iota_t^*R^2\pi_*\cO_{\cX}@>^{\cong}>>H^2(\cX_t,\cO_{\cX_t}),
\end{CD}
\eeq
where $\iota_t:t\hookrightarrow\Delta$. Because $\alpha_t\in H^{1,1}(\cX_t,\CC)$, we have $\varphi_t(\alpha_t)=0$.  Thus $\varphi(\alpha)=0$, so there exists a smooth $d$-closed $(1,1)$-form $\Omega$ such that $\{\Omega\}=\alpha$.

By the definition of $\mathcal{K}_{\cX,\RR}$, we have
\beq
\begin{CD}
0@>>>  PH_{\cX,\RR} @>>> \cE_{\cX,\RR}@>>>\mathcal{K}_{\cX,\RR}@>>>0.\nonumber
\end{CD}
\eeq
 Because $\cE_{\cX,\RR}$ is a soft sheaf,
\beq
R^k \pi_{*}\cE_{\cX,\RR}=0\nonumber
\eeq for $k\ge 1$, hence we have the following long exact sequence after pushforward.
\beq\label{sequence-3}
\begin{CD}
0@>>>  \pi_*PH_{\cX,\RR} @>>>  \pi_*\cE_{\cX,\RR}@>>>
\end{CD}
\eeq
\beq
\begin{CD}\nonumber
@>>> \pi_*\mathcal{K}_{\cX,\RR}@>>>R^1 \pi_*PH_{\cX,\RR} @>>>0
\end{CD}
\eeq

Let $\mathcal{M}$ be a sheaf satisfying
\beq\nonumber
\begin{CD}
0@>>>  \pi_*PH_{\cX,\RR} @>>>  \pi_*\cE_{\cX,\RR}@>>>\mathcal{M}@>>>0.
\end{CD}
\eeq
Since  $ \pi_*\cE_{\cX,\RR}$ is also soft, the cohomology groups satisfy  $H^k(\Delta,\mathcal{M})=0$ for $k\ge1$. By the short exact sequence
\beq\nonumber
\begin{CD}
0@>>> \mathcal{M}@>>> \pi_*\mathcal{K}_{\cX,\RR}@>>>R^1 \pi_*PH_{\cX,\RR} @>>>0,
\end{CD}
\eeq
we have
\beq\label{sequence-4}
\begin{CD}
\Gamma(\Delta, \pi_*\mathcal{K}_{\cX,\RR})@>>>\Gamma(\Delta,R^1 \pi_*PH_{\cX,\RR}) @>>> H^1(\Delta, \mathcal{M})=0\\
@V^{\cong}VV@V^{\cong}VV\\
\Gamma(\cX,\mathcal{K}_{\cX,\RR})@>>> H^1(\cX,PH_{\cX,\RR} )\nonumber.
\end{CD}
\eeq
Thus there exists a section $\Theta\in \Gamma(\cX,\mathcal{K}_{\cX,\RR})=\Gamma(\Delta,  \pi_*\mathcal{K}_{\cX,\RR})$ such that
\begin{eqnarray*}
\{\sqrt{-1}\partial\dbar\Theta\}=\{\Omega\}=\alpha\in H^2(\cX,\RR)
\end{eqnarray*}

\end{proof}

%

\begin{theo}\label{main-step}
Let $\pi\cln \cX\to \Delta$ be a proper smooth morphism from a complex manifold $\cX$ to the unit disc $\Delta$ such that all the fibers $\cX_t$ are K\"ahler for $t\in \Delta$.  Suppose that there exists an open dense subset  $W\subset \Delta$ such that $\cX_t\cong S$ for all $t\in W$. Then $\cX_t\cong S$ for all $t\in \Delta$ if $S$ is a non-uniruled K\"ahler manifold.
\end{theo}
\begin{proof}Fixed a point $t_0\in \Delta\setminus W$. There exists a class $\alpha\in H^2(\cX,\RR)$ on $X$ such that its restriction satisfies $\alpha_t\in H^{1,1}(\cX_t,\CC)$ and $\alpha_{t_0}$ is a K\"ahler class, as constructed in Lemma \ref{lemm4}. Let $\widetilde{\Delta}_{t_0}$ be a small disk around $t_0$, and $\pi_{\widetilde{\Delta}_{t_0}}: \cX_{\widetilde{\Delta}_{t_0}}=\pi^{-1}(\widetilde{\Delta}_{t_0})\to \widetilde{\Delta}_{t_0}$ be the restricted map. Then the restricted class $\widetilde{\alpha}=\widetilde{j}^*\alpha\in H^2(\cX_{\widetilde{\Delta}_{t_0}},\RR)$ satisfies $\widetilde{\alpha}_{t_0}=\alpha_{t_0}$ where $\widetilde{j}: \cX_{\widetilde{\Delta}_{t_0}}\to \cX$ is the canonical open embedding. By the proof of Lemma \ref{res1-2}, there exists a section $\Theta\in \Gamma(\cX_{\widetilde{\Delta}_{t_0}},\mathcal{K}_{\cX_{\widetilde{\Delta}_{t_0}},\RR})=\Gamma(\widetilde{\Delta}_{t_0},  \pi_*\mathcal{K}_{\cX_{\widetilde{\Delta}_{t_0}},\RR})$ such that
\begin{eqnarray*}
\{\sqrt{-1}\partial\dbar\Theta\}=\widetilde{\alpha}\in H^2(\cX_{\widetilde{\Delta}_{t_0}},\RR)
\end{eqnarray*}

Since $\alpha_{t_0}$ is a K\"ahler class, there exists a smooth function $\psi_{t_0}$  such that $$\sqrt{-1}(\partial\dbar\Theta|_{\cX_{t_0}}+\partial\dbar\psi_{t_0})$$ is a K\"ahler form on the  fiber $\cX_{t_0}$. Choose a small ball $\Delta_{t_0}\subset \widetilde{\Delta}_{t_0}$ centered at ${t_0}$, since
$$ \pi^{-1}(\Delta_{t_0})\cong \cX_{t_0}\times \Delta_{t_0}$$
as differentiable manifolds, we can extend the smooth function $\psi_{t_0}$ on $\cX_{t_0}$ canonically to a smooth function $\widetilde{\psi}_{t_0}$ on $ \pi^{-1}(\Delta_{t_0})$. After shrinking the ball, we may assume that $$\sqrt{-1}(\partial\dbar\Theta+\partial\dbar\widetilde{\psi}_{t_0})|_{\cX_s}$$
is a K\"ahler form on the fiber $\cX_s$ for each $s\in \Delta_{t_0}$. Thus the restriction $\pi_{\Delta_{t_0}}:\cX_{\Delta_{t_0}}\to \Delta_{t_0}$ is a K\"ahler morphism. Because $W\subset \Delta$ is an open dense subset, there exists a point $t_1\in W\cap \Delta_{t_0}$. Since $\cX_{t_1}$ is not an un-ruled manifold by assumption, all the fibers $\cX_t$ are not un-ruled for $t\in \Delta_{t_0}$ by \cite[Propositon 11]{Fu84}. Thus the pair $(\pi_{\Delta_{t_0}}:\cX_{\Delta_{t_0}}\to {\Delta_{t_0}},j^*\alpha)$ induces a holomorphic map $f:\Delta_{t_0}\to M$, where $j: \cX_{\Delta_{t_0}}\to \cX$.

 Applying the diagram (\ref{diagram-1}) to $\cX_{\Delta_{t_0}}$, and keeping the notations the same.
\begin{eqnarray*}
\xymatrix{
\cX_t \ar[r]^{\imath_t} \ar@<.5ex>[rd]^{\kappa_t}  &  \cX_{\Delta_{t_0}}\ar[d]^{T_{t_1}}\ar[r]^(0.35){T}\ar[rd]^(0.35){\pi}|\hole  & \cX_{t_1}\times\Delta_{t_0} \ar[dl]^(0.6){p_1}\ar[d]^{p_2}\\
& \cX_{t_1}\ar@<.5ex>[lu]^{\psi_t}  & \Delta_{t_0}
}
\end{eqnarray*}

Since all the fibers over $W\cap\Delta_{t_0}$ are mutually biholomorphic, from \cite{FG} we know that there exists a small neighborhood $V_{t_1}\subset \Delta_{t_0}\cap W$ and a holomorphic isomorphism
\beq
H:\cX_{V_1}:=\pi^{-1}(V_1)\cong \cX_{t_1}\times V_{t_1}
\eeq
such that $pr_2\circ H=\pi_{V_1}:\cX_{V_1} \to V_1$, where $pr_2:X_{t_1}\times V_{t_1}\to V_{t_1}$, and $H|_{\pi^{-1}(t_1)}=Id:X_{t_1}\to X_{t_1}$. Let  $H_{t_1}:=pr_1\circ H: \cX\to X_{t_1}$ and $h_t: \cX_t\to \cX_{t_1}$ be the composition of $\imath_t: \cX_t\to \cX$ with $H_{t_1}$. The biholomorphisms $\{h_t\}$ depend holomorphically on $t$. Let $\chi_t:X_{t_1}\to X_t$ be the inverse biholomorphism of $h_t$.

\begin{eqnarray*}
\xymatrix{
\cX_{t} \ar[r] \ar@<.5ex>[rd]^{h_t}  &  \cX_{V_1}\ar[d]^{H_{t_1}}\ar[r]^(0.35){H}\ar[rd]^(0.35){\pi_{V_1}}|\hole  & \cX_{t_1}\times V_1 \ar[dl]^(0.6){p_1}\ar[d]^{p_2}\\
& \cX_{t_1}\ar@<.5ex>[lu]^{\chi_t}  & V_1
}
\end{eqnarray*}

The composition $\phi_t:= h_t\circ\psi_t$
\begin{eqnarray*}
\xymatrix{
\cX_{t_1} \ar[r]^{\psi_t}  &  \cX_{t} \ar[r]^{h_t} & \cX_{t_1}
}
\end{eqnarray*}
 is a self-diffeomorphism of $\cX_{t_1}$ for $t\in V_1$. Let $Diff(\cX_{t_1})$ be the be the Fr$\acute{e}$chet Lie group of smooth diffeomorphisms of $\cX_{t_1}$, and let $Diff^0(\cX_{t_1})\subset Diff(\cX_{t_1})$ be the connected component of the identity map. Because $\phi_t$  depends smoothly on $t\in V_1$ and $\phi_{t_1}=Id$, we have $\phi_t\in Diff^0(\cX_{t_1})$. Therefore
 \beq\nonumber
 \alpha_{t_1}=\phi_t^*\alpha_{t_1}=\psi_t^*\left(h_t^*\alpha_{t_1}\right).
 \eeq
Thus $h_t^*\alpha_{t_1}=\left(\psi_t^*\right)^{-1}(\alpha_{t_1})$. By formula (\ref{iden-3}), we have $\alpha_t=h_t^*\alpha_{t_1}$. This means $h_t:\cX_t\to \cX_{t_1}$ induces an isomorphism of polarized pairs $(\cX_t,\alpha_t)$ and $(\cX_{t_1},\alpha_{t_1})$. So $f(V_{t_1})$ is a point. Because $f: \Delta_{t_0}\to M$ is holomorphic, the preimage of $f$ at a point is an analytic set. So $f$ is a constant map. Consequently, $\cX_t\cong S$ for $t\in \Delta_{t_0}$ if $S$ is a non-uniruled manifold.
\end{proof}

\subsection{The proof of the main theorem}

%

\begin{theo} Let $\cX$ be a complex manifold, and let $\pi\cln \cX\to \Delta=\{t\in\CC||t|<1\}$ be a proper smooth morphism such that all the fibers are K\"ahler manifolds. Suppose that there exists a small neighborhood $U_{t_1}$ of $t_1\in\Delta$ such that $\cX_t\cong \cX_{t_1}$ for all $t\in U_{t_1}$. Then  $\cX_t\cong \cX_{t_1}$ for all $t\in\Delta$, provided that $\cX_{t_1}$ is not an uniruled K\"ahler manifold.
\end{theo}
\begin{proof}{\bf Step 1.} We prove that there exist a class $\alpha\in H^2(\cX,\RR)$, and a connected open dense subset $W\subset \Delta$ containing  $U_{t_1}$ such that the restrictions $\alpha_t\in H^{1,1}(\cX_t,\CC)$ are K\"ahler classes for $t\in W$, and the pair $(\pi_{W}:\cX_W=\pi^{-1}(W)\to W ,\alpha)$ forms a family of polarized K\"ahler manifolds.

By Proposition~\ref{prop-1}, the morphism $\mu:C(\cX/\Delta)\to \Delta$
is proper. Let $\Lambda$ be the set of connected components of $C(\cX/\Delta)$, and
$$\Sigma:=\bigcup_{A\in \{C\in\Lambda:\mu(C)\subsetneq \Delta\}} \mu(A)\subseteq \Delta.$$ By \cite[Theorem]{F0} $C(\cX/\Delta)$ has at most countably many components.  Since $\mu$ is proper, $\Sigma$ is an union of countably many proper analytic subsets of $\Delta$.
Let $\widetilde{W}=\Delta\setminus \Sigma$, then $\widetilde{W}\subset \Delta$ is dense, and it is arcwise connected by piecewise smooth analytic arcs.  So $U_{t_1}\cap \widetilde{W}\neq \emptyset$, moreover we have $U_{t_1}\subset \widetilde{W}$ by Lemma \ref{Kahler}.

Let $\alpha\in H^2(\cX,\RR)$ be the class as constructed in Lemma \ref{lemm4}. Its restriction $\alpha_t\in H^{1,1}(\cX_t,\CC)$ for all $t\in\Delta$ and $\alpha_{t_1}$ is a K\"ahler class. Using an argument analogous to that of Demailly-P$\breve{\textrm{a}}$un in the proof of \cite[Theorem 0.8]{DP04}, we will prove that $\alpha_t$ are numerical positive for all $t\in \widetilde{W}$ as follows.

Since $\alpha_{t_1}$ is a K\"ahler class, we can take $\omega_{t_1}$ to be a K\"ahler form in the class of $\alpha_{t_1}$. Then
\beq\nonumber
(\alpha_{t_1})^k\cdot[Z]=\int_{Z} \omega_{t_1}^k>0.
\eeq
As discussed in \cite[$\S$ 5]{DP04}, there is a commutative diagram
$$
\xymatrix{
C^{k}({\cX}/\Delta) \ar[rr]^{\ \ \ \ p} \ar[dr]
                &  &    R^{2(n-k)}{\pi}_{*}\mathbb{Z}_{{\cX}} \ar[dl]    \\
                & \Delta                },
$$
where  the map $p$ sends a compact analytic cycle Z of dimension $k$ in $\cX_t$ to the cohomology class $[Z]\in H^{2(n-k)}(\cX_t,\mathbb{Z})$.
For each $k$-dimensional analytic cycle $Z\subset \cX_{t}$, where $1\le k\le n$, there exists a family of cohomology classes $\zeta_{Z,s}\in H^{2(n-k)}(\cX_{s},\CC)$  such that $\zeta_{Z,t}=[Z]$ and $\nabla^{2(n-k),GM} \zeta_{Z}=0$. When $Z$ varies in $\cX_{t}$, all the associated families together generate all classes of analytic cycles in $\cX_{t}$ for each $t \in \widetilde{W}$. So
\beq\nonumber
\frac{d}{ds}((\alpha_{s})^k\cdot\zeta_{Z,s})=k\nabla^{2,GM}\alpha_{s}\cdot(\alpha_{s})^{k-1}\cdot\zeta_{Z,s}=0\quad \mbox{for}\quad s\in \Delta.
\eeq
Thus $(\alpha|_{s})^k\cdot\zeta_{Z,s}$ is a constant function of $s\in \Delta$, and
\beq\label{equ-2}\nonumber
\int_{Z_{t}}(\alpha_{t})^k=(\alpha_{t})^k\cdot\zeta_{Z,t}=(\alpha_{t_0})^k\cdot\zeta_{Z,t_0}>0.
\eeq
Therefore $\alpha_{t}$ is numerically positive for all $t\in \widetilde{W}$.

By formula (\ref{iden-4}), the constant sheaves $H_{\Delta}^{p,q}:=\Delta\times H^{p,q}(\cX_{t_1},\CC)\subset R^k\pi_*\CC$ form a series of constant subsheaves. By Hodge decomposition, we know that
 \beq
 \nabla^{2,GM}=\text{diag} \{\nabla^{2,0,GM},\nabla^{1,1,GM},\nabla^{0,2,GM}\},
 \eeq
where $\nabla^{p,q,GM}$ is the restriction of $\nabla^{2,GM}$ on the subsheaf $\cup_{t\in\Delta}H^{p,q}(\cX_t,\CC)$. Let $s(t)=\alpha_t$ be the section of $R^2\pi_*\CC$ induces by $\alpha$. Then $\nabla^{1,1,GM} s= \nabla^{2,GM}s=0$. Thus $s$ is a flat section with respect to $\nabla^{1,1,GM}$. Because $\alpha_{t_1}$ is a K\"ahler class on $\cX_{t_1}$, the class $\alpha_{t}$ is K\"ahler for each $t\in \widetilde{W}$ by \cite[Theorem 0.9]{DP04}, which asserts that the K\"ahler cone is invariant under parallel transport with respect  $\nabla^{1,1,GM}$.

 Since $\pi$ is smooth and $\Delta$ is contractible, from spectral sequence we have
\beq\nonumber
\begin{CD}
 @>>>H^1(\cX,PH_{\cX,\RR})@>>>H^2(\cX,\RR) @>^{\varphi}>> H^2(\cX,\cO_\cX)\\
 @.@.@V^{\cong}VV@V^{\cong}VV\\
 @.@.H^0(\Delta,R^2\pi_*\RR)@>^{\widetilde{\varphi}}>>H^0(\Delta,R^2\pi_*\cO_\cX)
\end{CD}
\eeq
 Because $\alpha_t\in H^{1,1}(\cX_t,\CC)$ for all $t\in\Delta$, we have $\widetilde{\varphi}_t(\alpha_t)=0$ for $t\in\Delta$. Thus $\varphi(\alpha)=0$. By the proof of Lemma \ref{res1-2}, there exists a section $\Theta\in \Gamma(\cX,\mathcal{K}_{\cX,\RR})$ such that
\begin{eqnarray*}
\{\sqrt{-1}\partial\dbar\Theta\}=\alpha\in H^2(\cX,\RR).
\end{eqnarray*}

 For each $t\in \widetilde{W}$, $\alpha_{t}$ is a K\"ahler class. Hence there exists a smooth function $\psi_t$ on $\cX_t$ such that $\sqrt{-1}\partial\dbar\Theta|_{\cX_t}+\partial\dbar\psi_t$ is a K\"ahler form. Choosing a small ball $V_t$ centered at $t$, then $\pi^{-1}(V_t)\cong \cX_t\times V_t$ as differentiable manifolds. The smooth function $\psi_t$ can be canonically extended to a smooth function $\widetilde{\psi_t}$ on $\pi^{-1}(V_t)$. After shrinking the ball, we can assume that all the forms $(\sqrt{-1}\partial\dbar\Theta+\partial\dbar\widetilde{\psi_t})|_{\cX_s}$ on  $\cX_s$ are  K\"ahler  for $s\in V_t$. So
$$W:=\{t\in\Delta:\alpha_{t}\mbox{ is a K\"ahler class}\}$$ is an open subset of $\Delta$ whose complement has at most countably many points, moreover $W$ is a pseudoconvex domain. The pair $(\pi_{W}:\cX_W=\pi^{-1}(W)\to W ,\alpha)$ forms a family of polarized K\"ahler manifolds.

{\bf Step 2.} We prove $\cX_t\cong \cX_{t_1}$ for all $t\in W$ in this step.

By \cite{FG} there exists a small neighborhood $V_{t_1}\subset U_{t_1}$ and a holomorphic isomorphism
\beq
H:\cX_{V_1}:=\pi^{-1}(V_1)\cong \cX_{t_1}\times V_{t_1}
\eeq
such that $pr_2\circ H=\pi_{V_1}:\cX_{V_1} \to V_1$, where $pr_2:X_{t_1}\times V_{t_1}\to V_{t_1}$, and $H|_{\pi^{-1}(t_1)}=Id:X_{t_1}\to X_{t_1}$. Let  $H_{t_1}:=pr_1\circ H: \cX\to X_{t_1}$ and $h_t: \cX_t\to \cX_{t_1}$ be the composition of $\imath_t: \cX_t\to \cX$ with $H_{t_1}$. The biholomorphisms $\{h_t\}$
 depend holomorphically on $t$. Let $\chi_t:X_{t_1}\to X_t$ be the inverse biholomorphism of $h_t$.

\begin{eqnarray*}
\xymatrix{
\cX_{t} \ar[r] \ar@<.5ex>[rd]^{h_t}  &  \cX_{V_1}\ar[d]^{H_{t_1}}\ar[r]^(0.35){H}\ar[rd]^(0.35){\pi_{V_1}}|\hole  & \cX_{t_1}\times V_1 \ar[dl]^(0.6){p_1}\ar[d]^{p_2}\\
& \cX_{t_1}\ar@<.5ex>[lu]^{\chi_t}  & V_1
}
\end{eqnarray*}

Because $\cX_{t_1}$ is not an un-ruled manifold, all the fibers $\cX_t$ are not un-ruled for $t\in W$ by \cite[Propositon 11]{Fu84}. Thus the pair $(\pi_{W}:\cX_{W}\to W,\alpha)$ induces a holomorphic map $f:W\to M$, where $j: \cX_{W}\to \cX$.

The composition $\phi_t:= h_t\circ\psi_t$
\begin{eqnarray*}
\xymatrix{
\cX_{t_1} \ar[r]^{\psi_t}  &  \cX_{t} \ar[r]^{h_t} & \cX_{t_1}
}
\end{eqnarray*}
 is a self-diffeomorphism of $\cX_{t_1}$ for $t\in V_1$. Let $Diff(\cX_{t_1})$ be the be the Fr$\acute{e}$chet Lie group of smooth diffeomorphisms of $\cX_{t_1}$, and let $Diff^0(\cX_{t_1})\subset Diff(\cX_{t_1})$ be the connected component of the identity map. Because $\phi_t$ is smooth dependent on $t\in V_1$ and $\phi_{t_1}=Id$, so $\phi_t\in Diff^0(\cX_{t_1})$. Therefore
 \beq\nonumber
 \phi_t^*\alpha_{t_1}=\psi_t^*\left(h_t^*\alpha_{t_1}\right)=\alpha_{t_1}.
 \eeq
Thus $h_t^*\alpha_{t_1}=\left(\psi_t^*\right)^{-1}(\alpha_{t_1})$. By formula (\ref{iden-3}), we have $\alpha_t=h_t^*\alpha_{t_1}$. This means that $h_t:\cX_t\to \cX_{t_1}$ induces an isomorphism of polarized pairs $(\cX_t,\alpha_t)$ and $(\cX_{t_1},\alpha_{t_1})$. It implies that $f(V_{t_1})$ is a point. Because $f: W\to M$ is holomorphic, the preimage of $f$ at a point is an analytic subset. So $f$ is a constant map. Thus $\cX_t\cong \cX_{t_1}$ for $t\in W$ if $\cX_{t_1}$ is a non-uniruled manifold.

{\bf Step 3.} In this step, we will finish the proof.

By Theorem \ref{main-step},  $\cX_t\cong \cX_{t_1}$ for all $t\in\Delta$, provided that $\cX_{t_1}$ is not an uniruled K\"ahler manifold.

\end{proof}

\bibliographystyle{amsplain}

\end{document}